\documentclass[11pt,reqno]{amsart}

\usepackage{subfigure}
\usepackage{amsthm}
\usepackage{paralist}
\usepackage{enumitem}
\usepackage{pifont}
\usepackage[utf8]{inputenc}
\usepackage[cal=boondox]{mathalfa}

\newtheorem{lemma}{Lemma}
\newtheorem{theorem}{Theorem}[section]
\newtheorem{definition}{Definition}

\newtheorem{remark}{Remark}[section]

\newcommand{\xd}{\textrm{d}}

\title[From nonlocal to local Cahn-Hilliard]{From nonlocal to local Cahn-Hilliard equation}
\author[Stefano Melchionna]{Stefano Melchionna}
\address{Institut f\"ur Mathematik, University of Vienna, Oskar-Morgenstern-Platz 1, 1090 Vienna, Austria}
\email{melchionna.s90@gmail.com}
\author[Helene Ranetbauer]{Helene Ranetbauer}
\address{Institut f\"ur Mathematik, University of Vienna, Oskar-Morgenstern-Platz 1, 1090 Vienna, Austria}
\email{helene.ranetbauer@univie.ac.at}
\author[Luca Scarpa]{Luca Scarpa}
\address{Institut f\"ur Mathematik, University of Vienna, Oskar-Morgenstern-Platz 1, 1090 Vienna, Austria}
\email{luca.scarpa@univie.ac.at}
\author[Lara Trussardi]{Lara Trussardi}
\address{Institut f\"ur Mathematik, University of Vienna, Oskar-Morgenstern-Platz 1, 1090 Vienna, Austria}
\email{lara.trussardi@univie.ac.at}

\begin{document}

\begin{abstract}
In this paper we prove the convergence of a nonlocal version of the Cahn-Hilliard equation to its local counterpart as the nonlocal convolution kernel is scaled using suitable approximations of a Dirac delta in a periodic boundary conditions setting. This convergence result strongly relies on the dynamics of the problem. More precisely, the $H^{-1}$-gradient flow structure of the equation allows to deduce uniform $H^1$ estimates for solutions of the nonlocal Cahn-Hilliard equation and, together with a Poincar\'e type inequality by Ponce, provides the compactness argument that allows to prove the convergence result.
\end{abstract}

\maketitle

%%%%%%%%%%%%%%%%%%%%%%%%%%%%%%%%%%%%%%%%%%%%%%%%%%%%%%%%%%%%%%%%%%%%%%%%%%%%%%%%%%%%%%%%%%%%%%%%
\section{Introduction}
\label{sec:intro}
\indent 
%%%%%%%%%%%%%%%%%%%%%%%%%%%%%%%%%%%%%%%%%%%%%%%%%%%%%%%%%%%%%%%%%%%%%%%%%%%%%%%%%%%%%%%%%%%%%%%%
The Cahn-Hilliard equation~\cite{Ca, CH} is widely used in the study of phase field models as well as diffuse interface theory and it was developed to describe the evolution of the concentration of two components in a binary fluid. This equation typically arises in connection with phase transitions which occur when a substance changes from a state (e.g. solid, liquid, or gas) into a different one exhibiting different properties.

There are several examples for this kind of phenomena: the condensation of water drops in mist, a homogeneous molten binary alloy that is rapidly cooled, mixtures in general (two metallic, polymer or glassy components) as well as pattern formation. However, the Cahn-Hilliard equation is also relevant in many other applications like image processing~\cite{capuzzo2002area}, population dynamics~\cite{cohen1981generalized} or even the formation of Saturn rings~\cite{tremaine2003origin}.\\
In the literature two types of models have been proposed to study phase transitions: sharp-interface and phase-field models. Where sharp-interface models describe the interface as a $(d-1)$-dimensional hypersurface, phase-field models replace the sharp interface by a thin transition region in which a mixture of the two components is present. 

Originally, the Cahn-Hilliard equation was introduced for modelling the phenomena of spinodal decomposition, i.e. the loss of mixture homogeneity and the formation of pure phase regions, and coarsening dynamics, which is the aggregation of pure phase regions into larger domains. The model exhibits a gradient-flow structure (in the $H^{-1}$-metric) in terms of the free energy functional given by, cf.~\cite{CH},
\begin{equation}
\label{eq:enfunCH}
E_{CH}(u(x))=\int_{\Omega}\Bigl(\frac{\tau^{2}}{2}|\nabla u(x)|^{2}+F(u(x))\Bigr) \,\xd x.
\end{equation}
Note that $\tau$ is a small positive parameter related to the transition region thickness. In this paper, $\Omega$ denotes the $d$-dimensional ($d=3$) flat torus and $F$ is a double well potential with two global minima representing the pure phases and with second derivatives bounded from below. 
The most natural choice for $F$ is the free energy density 
obtained through the principles of statistical mechanics,
defined by
\[
  F_0(s):=\theta_0s(1-s) + \theta[s\log s + (1-s)\log(1-s)]\,,
  s\in(0,1)\,,
\]
where the constants $0<\theta<\theta_0$ are related to 
the temperature of the system and the Boltzmann constant.
A usual polynomial approximation
of $F_0$ is in the form
\[
  F_P(s):=A_1s^4 - A_2 s^2\,,
\]
where $A_1$ and $A_2$ are positive constants depending on $\theta_0$ and $\theta$.
The corresponding evolution problem is given by the $H^{-1}$-gradient flow with respect to the energy functional~\eqref{eq:enfunCH}
\begin{align}\label{eq:CH}
\begin{aligned}
\frac{\partial u}{\partial t}+\nabla\cdot J_{CH} & =0\text{,}\\ 
J_{CH} & =-\mu(u)\nabla v_{CH}\text{,}\\
v_{CH} & =\frac{\delta E_{CH}(u)}{\delta u}=-\tau^{2}\Delta u+F^{\prime}(u)\text{.}
\end{aligned}
\end{align}
The function $\mu(\cdot)$ in~\eqref{eq:CH} is known as mobility. 

\noindent
Even though in the existing literature the Cahn-Hilliard equation has been studied intensively and also successfully, it still cannot be rigorously derived as a macroscopic limit of microscopic models for interacting particles. 
A nonlocal version of the equation, proposed by Giacomin and Lebowitz~\cite{GL}, attracted great interest in recent years. They considered the hydrodynamic limit of such a microscopic model and derived a nonlocal energy functional of the form
\begin{align} \label{eq:enfunGL}
E_{NL}(u(x))&=\frac{1}{4}\int_{\Omega}\int_{\Omega}K(x,y)(u(x)-u(y))^{2}\mathrm{d}
x\xd y+ \int_{\Omega}F(u(x))\xd x,
\end{align}
where $K(x,y)$ is a positive and symmetric convolution kernel.
The associated evolution problem is a nonlocal variant of the Cahn-Hilliard system
\begin{align}\label{eq:NLCH}
\begin{aligned}
\frac{\partial u}{\partial t}+\nabla\cdot J_{NL} & =0\text{,}\\ 
J_{NL} & =-\mu(u)\nabla v_{NL}\text{,}\\
v_{NL} & =\frac{\delta E_{NL}(u)}{\delta u}=(K*1) u-K* u+ F'(u) \text{,}
\end{aligned}
\end{align}
where $(K*1)(x):=\int_{\Omega}K(x,y)\xd y$ and $(K*u)(x):=\int_\Omega K(x,y)u(y)\, \xd y$.
Being a choice often considered in the existing literature, we take a constant mobility and, for simplicity, we set $\mu=1$ both in~\eqref{eq:CH} and~\eqref{eq:NLCH}.

Note that the local Cahn-Hilliard system~\eqref{eq:CH} is a fourth order PDE, whereas the nonlocal one~\eqref{eq:NLCH} is an integro-differential second order parabolic equation. However, they share a lot of fundamental features ranging from the underlying gradient flow structure, the lack of comparison principles, the separation of the solution from the pure phases~\cite{CMZ, LP2, gal2017nonlocal} to the long time behaviour~\cite{LP1}.
Moreover, both energy functionals allow the same $\Gamma$-limit for vanishing interface thickness (see~\cite{AB, MM} and~\cite{GNS, Le} for the sharp interface limit of the local Cahn-Hilliard equation).

In this paper the local concentration of one of the two components is represented by a real valued function $u=u(x)$. The pure phases are chosen as $0$ and $1$. Compared to sharp-interface models, we neither have to worry about complicated boundary conditions across the interface nor being concerned with regularity issues.\\
We are interested in proving convergence of weak solutions of the nonlocal Cahn-Hilliard equation~\eqref{eq:NLCH} to weak solutions of the local version~\eqref{eq:CH} as the convolution kernel $K$
is scaled by using suitable approximations of a Dirac delta. 
More precisely, we consider the following family of convolution kernels, parametrized by a small positive parameter $\varepsilon$:
\begin{equation}\label{eq:kernel}
K_\varepsilon(x,y)=\frac{\varepsilon^{-d}}{|x-y|^2}\rho\left(\frac{|x-y|}{\varepsilon}\right),
\end{equation}
where $\rho:\mathbb{R}^+\to \mathbb{R}^+$ and $\rho$ is a nonnegative, decreasing, continuous function with compact support such that $\int_\Omega \rho(|z|)\, \xd z=\int_{\mathbb{R}^d}\rho(|z|)\, \xd z$.
A classical choice of $\rho$ is 
\[
  \rho(r):=\begin{cases}
  c e^{-\frac{1}{r_0^2-r^2}} \quad&\text{for } 0\leq r<r_0\,,\\
  0 \quad&\text{for } r\geq r_0\,,
  \end{cases}
\]
where $c$ is a renormalization constant and $r_0>0$.
With this choice, the convolution kernel that we consider is of the form
\[
  K_\varepsilon(x,y)=\frac{\rho_\varepsilon(|x-y|)}{|x-y|^2}\,,
\]
where $(\rho_\varepsilon)_\varepsilon$ is a suitable
family of mollifiers on $\mathbb{R}^d$.
It is well known that, with this choice for the kernel, the nonlocal energy functional $E_{NL} $ converges to the local one $E_{CH}$ pointwise in $H^1(\Omega)$, provided appropriate growth conditions on the potential $F$, see~\cite{BBM01,BN16}. Indeed the local term $\tau |\nabla u|^2$ can be obtained as the formal limit of the corresponding nonlocal terms with kernel~\eqref{eq:kernel}
as $\varepsilon \to 0$, where $\tau:=\frac{1}{2}\int_\Omega \rho(|z|) \xd z$, see~\cite{krejci2007nonlocal}. 

Note that, by denoting
\[
E_\varepsilon(u_\varepsilon(x)) := \tilde{E}_\varepsilon(u_ \varepsilon(x))+\int_\Omega F(u_\varepsilon(x))\xd x\,,
\]
where
\[
\tilde{E}_\varepsilon(u_ \varepsilon(x)):=
\frac{1}{4}\int_{\Omega}
\int_{\Omega}K(x,y)(u_\varepsilon(x)
-u_\varepsilon(y))^{2}\mathrm{d}
x\xd y\,,
\]
with kernel $K_\varepsilon$ as in~\eqref{eq:kernel} and $u_\varepsilon$ is the solution to the corresponding nonlocal Cahn-Hilliard equation, it is possible to show the uniform boundedness of the nonlocal energies $E_{\varepsilon}(u_\varepsilon)$. Taking advantage of a result by Ponce~\cite{ponce2004estimate} allows to obtain strong convergence of a (not relabelled) subsequence $u_\varepsilon$ in the $L^2$-topology to a limit $u \in H^1(\Omega )$. However, it is not clear whether strong convergence in $L^2$ suffices to pass to the limit. 
Moreover, the $\Gamma$-convergence cannot directly be deduced from the pointwise convergence of $E_{\varepsilon}$ in $H^1(\Omega)$, since the energy functional is non convex and the domain of $E_{\varepsilon}$ is larger than $H^1(\Omega)$ (it is $L^p(\Omega)$ with $p$ depending on the growth of the potential $F$) or, in other words, because of the lack of coercivity of $E_{\varepsilon}$ in $H^1(\Omega)$. Nevertheless, Ponce proved a result on $\Gamma$-convergence for the energy functionals, see~\cite{Po04}. Trying to approach the problem in the evolutionary setting following the method in the spirit of Sandier and Serfaty~\cite{SS04,Se11} is by far not trivial and beyond the goal of this paper. 

In order to overcome this problem we argue as follows. First we note that for every positive $\varepsilon$ solutions to the associated $H^{-1}$-gradient flow (namely of the nonlocal Cahn-Hilliard equation~\eqref{eq:NLCH}) belong to $H^1(\Omega)$ almost everywhere in time, although solutions to the stationary problem, i.e. minimizers of $E_\varepsilon$, cannot be guaranteed to belong to $H^1(\Omega)$. Moreover, by suitable choices of the test functions in the weak formulation of the equation~\eqref{eq:NLCH}, using a Poincar\'e-type inequality derived in~\cite{ponce2004estimate} we can prove uniform-in-$\varepsilon$ bounds on $u_\varepsilon$ in $H^1(\Omega)$ in the case of periodic boundary conditions. Furthermore, 
suitably applying a compactness inequality, we are able to prove also strong convergence in $H^1(\Omega)$, which allows us to pass to the weak limit in the equation. We finally note that, by using the uniqueness of solutions, the limit $u=\lim_{\varepsilon \to 0} u_\varepsilon$ can be proved to enjoy additional regularity ($H^2$ in space) and hence to be a weak solution to the local Cahn-Hilliard equation~\eqref{eq:CH}.\\

%%%%%%%%%%%%%%%%%%%%%%%%%%%%%%%%%%%%%%%%%%%%%%%%%%%%%%%%%%%%%%%%%%%%%%%%%%%%%%%%%%%%%%%%%%%%%%%%
\section{Preliminaries and Main Result} \label{sec:solutions}
\indent 
%%%%%%%%%%%%%%%%%%%%%%%%%%%%%%%%%%%%%%%%%%%%%%%%%%%%%%%%%%%%%%%%%%%%%%%%%%%%%%%%%%%%%%%%%%%%%%%%

In this paper we are interested in the convergence of solutions of the nonlocal Cahn-Hilliard equation~\eqref{eq:NLCH} to solutions of the local version~\eqref{eq:CH} in a periodic setting.\\ 

We start by enlisting our assumptions.
\begin{description}
\item [H1] $\Omega$ is the $d$-dimensional $(d=3)$ flat torus.
\item [H2] 
The kernel $K_\varepsilon$ is defined as in~\eqref{eq:kernel}, i.e.
\[
K_\varepsilon(x,y)=\tilde{K}_\varepsilon(x-y)=
\frac{\varepsilon^{-d}}{|x-y|^2}
\rho\left(\frac{|x-y|}{\varepsilon}\right),
\]
with
\[
\tau:=\frac{1}{2} \int_\Omega \rho(|z|)\, \xd z=1 
\]
and  $\rho:\mathbb{R}^+\to\mathbb{R}^+$ is a sufficiently smooth, nonnegative, decreasing and continuous function with compact support such that $\int_\Omega \rho(|z|)\, \xd z=\int_{\mathbb{R}^d}\rho(|z|)\, \xd z$.
\item[H3] $F\in C^2(\mathbb{R})$ is a double well potential with two global minima at $0$ and $1$ such that $F(0)=0=F(1)$, and 
\begin{align}
F''(r)&\geq -B_1\qquad\forall\,r\in\mathbb{R}\,,\\
|F''(r)|&\leq B_2(|r|^2+1)\qquad\forall\,r\in\mathbb{R},
\end{align}
for some constant $B_1,B_2>0$.
\item[H4] The initial data $u_{0,\varepsilon} \in L^2(\Omega)$ converges strongly in $L^2(\Omega)$ to the limit $u_0\in H^1(\Omega)$ and satisfies $E_\varepsilon(u_{0,\varepsilon}), E(u_0) \leq C_0$ for some constant $C_0>0$ independent of $\varepsilon$.
For example, if $(u_{0,\varepsilon})_\varepsilon\subset H^1(\Omega)$ and 
$u_{0,\varepsilon}\rightharpoonup u_0$ weakly in $H^1(\Omega)$
this condition is satisfied.
\end{description}

\begin{remark}
Note that for dimension $d=3$, the kernel $K_\varepsilon\in L^1(\Omega)$, so all the convolution terms appearing in the nonlocal equation are well defined. If $d=2$, the form of the convolution kernel implies that $K_\varepsilon$ does not belong to $L^1(\Omega)$. In particular $(K_\varepsilon*1)(x)=\int_\Omega \frac{\rho_\varepsilon(|z|)}{|z|^2}\, \xd z=\infty$ for any $x\in \Omega\,$ for $d=2$. However, even for $d=2$ the formulation of the nonlocal equation can be made rigorous by introducing a linear operator representing the nonlocality $\varphi \mapsto (K_\varepsilon*1)\varphi-K_\varepsilon*\varphi$ (see \cite{davoli2019degenerate} for details). In our paper, we restrict ourselves to dimension $d=3$ to avoid technicalities. 
\end{remark}

\begin{remark}
  Note that assumption \textbf{(H3)} is satisfied by the 
  $4$-th order polynomial double-well potential $F_P$
  mentioned in the introduction.
\end{remark}

\bigskip

Before stating our main result, let us recall the notion of weak solution to the nonlocal and local Cahn-Hilliard equation.

\begin{definition}[Weak solution to the nonlocal Cahn-Hilliard equation]\label{def:wsnl}
Let $\varepsilon >0 $ and $T>0$ be fixed. We define $u_\varepsilon$ to be a \emph{weak solution} to the nonlocal Cahn-Hilliard equation~\eqref{eq:NLCH} on $[0,T]$ associated with the initial datum $u_{0,\varepsilon} \in L^2(\Omega)$ if 
\[
u_\varepsilon \in H^1(0,T;(H^1(\Omega))^\ast) \cap L^2(0,T;H^1(\Omega)),
\]
satisfies 
\begin{align}\label{weak_NL}
\langle\partial_t u_\varepsilon, \varphi\rangle_{(H^1(\Omega))^*,H^1(\Omega)} + \int_\Omega \nabla[( K_\varepsilon * 1) u_\varepsilon -K_\varepsilon * u_\varepsilon +F^{\prime}(u_\varepsilon)]\cdot \nabla \varphi\, \xd x =0
\end{align}
for all $\varphi \in H^1(\Omega)$,
almost everywhere in $(0,T)$, and $u_\varepsilon (0) = u_{0,\varepsilon}$.
\end{definition}

\begin{definition}[Weak solution to the local Cahn-Hilliard equation]\label{def:wsl}
Let $T>0$ be fixed. We define $u$ to be a \emph{weak solution} to the Cahn-Hilliard equation~\eqref{eq:CH} on $[0,T]$ associated with the initial datum $u_0 \in H^1(\Omega)$ if
\[
u \in H^1(0,T;(H^1(\Omega))^\ast) \cap L^2(0,T;H^2(\Omega)),
\]
satisfies 
\begin{equation}\label{eq def weak sol ch}
\langle\partial_t u, \varphi\rangle_{(H^1(\Omega))^*,H^1(\Omega)}+ \int_\Omega \Delta u \Delta \varphi \, \xd x -\int_\Omega F^{\prime}(u) \Delta \varphi \, \xd x = 0
\end{equation}
for all $\varphi \in H^2(\Omega)$,
almost everywhere in $(0,T)$, and $u (0) = u_{0}$.
\end{definition}
\begin{remark}
Note that since $u_0\in H^1(\Omega)$, there is a unique solution $u$ to the local Cahn-Hilliard equation, which also satisfies the regularity $L^\infty(0,T;H^1(\Omega))\cap L^2(0,T;H^3(\Omega))$.\\
Moreover, let us point out that the solutions to both local and nonlocal Cahn-Hilliard equations do not satisfy $0\leq |u(x,t)|\leq 1$ for all $(x,t) \in \Omega \times (0,T)$.
\end{remark}
Existence and uniqueness of weak solutions to the local equation is well known with different choices for the boundary conditions: we mention, between others ~\cite{BE91, Ell89, EZ86, Mi17,NST89, NC08}. 
On the other hand, existence of solutions to the nonlocal equation is proven in the literature for kernels in $W^{1,1}(\Omega)$ (or satisfying analogous assumptions) ~\cite{BH, BH2, GZ, GL, Ji04}. Indeed, typically, in the process of deriving a-priori estimates, the convolution product is differentiated and the derivatives are placed on the term containing the kernel. This allows controlling $W^{1,p}$ norms of the convolution term with $W^{1,1}$ norms of the kernel and $L^p$ norms of $u$. In our case, however, taking advantage of the periodic boundary conditions and of the specific form of $K_\varepsilon$, $H^1(\Omega)$-estimates for $u$ can be derived without differentiating the kernel $K_\varepsilon$ (see Section \ref{ssec:estimates} for details). 
Additional estimates on the chemical potential can be derived from classical energy estimates for the Cahn-Hilliard equation. Thus, existence of solutions to \eqref{eq:NLCH} can be easily proved with our assumptions, e.g., via a classical Galerkin approximation scheme. A detailed proof of the existence results in a more general framework is given in \cite{davoli2019degenerate} (also for dimension $d=2$). 

The local as well as the nonlocal Cahn-Hilliard equation have been largely studied in the last years concerning for example qualitative properties~\cite{CMZ,gal2017nonlocal,LP2}, numerical aspects~\cite{BE92,GLWW14}, long-time behaviour~\cite{CMZ14, CG14,LP1} or asymptotics~\cite{AB, GNS, MM} with different kinds of boundary conditions and different potentials.\\

We now state our main result; the proof is shown in the following section.

\begin{theorem}[Convergence for weak solution in the periodic setting] \label{main thm}
Let assumptions \emph{\textbf{H1--H4}} be satisfied. Then, for $B_1$ sufficiently small, the weak solution $u_\varepsilon$ to the nonlocal Cahn-Hilliard equation~\eqref{eq:NLCH} converges strongly in $L^2(0,T;H^1(\Omega))\cap C([0,T]; (H^1(\Omega))^*)$ and weakly in $H^1(0,T;(H^1(\Omega))^*)$, to the weak solution of the local Cahn-Hilliard equation~\eqref{eq:CH}.
\end{theorem}

\begin{remark}
Note that the assumption ``$B_1$ sufficiently small" means $B_1<\frac{1}{2C_p}$, where $C_p$ is a constant depending on the dimension $d$ and on the domain $\Omega$ coming from a Poincar\'e type inequality derived in~\cite{ponce2004estimate} (see inequality~\eqref{ineq_ponce} below).
\end{remark}

\begin{remark}
\emph{\textbf{Theorem~\ref{main thm}}} is also valid for higher dimensions provided existence of solutions to the corresponding equations and assuming appropriate growth conditions on the potential $F$, i.e. modifying assumption \emph{\textbf{H3}} in order to be able to pass to the limit. In particular the new condition on $F$ would be 
\[
|F'(u)|^2\leq C(|u|^{2^*-2}+1)
\]
where $2^*$ is the critical Sobolev exponent for the embedding $H^1(\Omega)\hookrightarrow L^{2^*}(\Omega)$.
\end{remark}

%%%%%%%%%%%%%%%%%%%%%%%%%%%%%%%%%%%%%%%%%%%%%%%%%%%%%%%%%%%%%%%%%%%%%%%%%%%%%%%%%%%%%%%%%%%%%%%%%%%%%%%%%
\section{Proof of Convergence Result} \label{sec:periodic}
\indent 
%%%%%%%%%%%%%%%%%%%%%%%%%%%%%%%%%%%%%%%%%%%%%%%%%%%%%%%%%%%%%%%%%%%%%%%%%%%%%%%%%%%%%%%%%%%%%%%%%%%%%%%%
In this section, we prove Theorem~\ref{main thm}. The proof is divided into several steps. 

%%%%%%%%%%%%%%%%%%%%%%%%%%%%%%%%%%%%%%%%%%%%%%%%%%%%%%%%%%%%%%%%%%%%%%%%%%%%%%%%%%%%%%%%%%%%%%%%%%%%%%%%
\subsection{Uniform Estimates} \label{ssec:estimates}
%%%%%%%%%%%%%%%%%%%%%%%%%%%%%%%%%%%%%%%%%%%%%%%%%%%%%%%%%%%%%%%%%%%%%%%%%%%%%%%%%%%%%%%%%%%%%%%%%%%%%%%%

We start by choosing $\varphi=u_\varepsilon$ as a test function in equation~\eqref{weak_NL}, obtaining
\begin{align*}
\frac{1}{2}\frac{\xd}{\xd t}\|u_\varepsilon\|_{L^2(\Omega)}^2+\int_\Omega \nabla \left[(K_\varepsilon*1)u_\varepsilon-K_\varepsilon*u_\varepsilon+F'(u_\varepsilon)\right]\cdot \nabla u_\varepsilon\, \xd x&=0.
\end{align*}
As a consequence of the periodic boundary conditions $K_\varepsilon*1$ is constant over the domain. Moreover, $\nabla (K_\varepsilon*u_\varepsilon)=K_\varepsilon*\nabla u_\varepsilon=\nabla K_\varepsilon*u_\varepsilon$. Thus, the equation above reduces to
\begin{align*}
\frac{1}{2}\frac{\xd}{\xd t}\|u_\varepsilon\|_{L^2(\Omega)}^2+\int_\Omega [ (K_\varepsilon*1 )|\nabla u_\varepsilon|^2-(K_\varepsilon*\nabla u_\varepsilon)\cdot\nabla u_\varepsilon+F''(u_\varepsilon)|\nabla u_\varepsilon|^2 ]\, \xd x&=0
\end{align*}
and we rewrite
\begin{align*}
&\int_\Omega [(K_\varepsilon*1) |\nabla u_\varepsilon|^2-(K_\varepsilon*\nabla u_\varepsilon)\cdot\nabla u_\varepsilon]\, \xd x\\
&\qquad=\frac{1}{2}\int_\Omega \int_\Omega K_\varepsilon(x,y)|\nabla u_\varepsilon(x)-\nabla u_\varepsilon(y)|^2\, \xd x\, \xd y.
\end{align*}

Using the definition of the kernel~\eqref{eq:kernel} we obtain
\begin{align}\label{basic_estimate}
&\frac{1}{2}\frac{\xd}{\xd t}\|u_\varepsilon\|_{L^2(\Omega)}^2+\frac{1}{2}\int_\Omega\int_\Omega \frac{\varepsilon^{-d}}{|x-y|^2}\rho\left(\frac{|x-y|}{\varepsilon}\right)|\nabla u_\varepsilon(x)-\nabla u_\varepsilon(y)|^2 \,\xd x\, \xd y \notag \\
&+\int_\Omega F''(u_\varepsilon)|\nabla u_\varepsilon|^2\, \xd x=0.
\end{align}
From \textbf{H3} and~\eqref{basic_estimate}, it follows that
\begin{align}
\frac{1}{2}\frac{\xd}{\xd t}\|u_\varepsilon\|_{L^2(\Omega)}^2&+\frac{1}{2}\int_\Omega\int_\Omega \frac{\varepsilon^{-d}}{|x-y|^2}\rho\left(\frac{|x-y|}{\varepsilon}\right)| \nabla u_\varepsilon(x)-\nabla u_\varepsilon(y)|^2 \,\xd x\, \xd y \notag
\\
\label{est1}
&\leq B_1 \|\nabla u_\varepsilon\|_{L^2(\Omega)}^2.
\end{align}

Using the Poincar\'e-type inequality in~\cite{ponce2004estimate}, we get for sufficiently small $\varepsilon$, recalling that $\overline{ \nabla u_\varepsilon}=0$,
\begin{align}\label{ineq_ponce}
\|\nabla u_\varepsilon \|_{L^2(\Omega)}^2&\leq C_p\int_\Omega\int_\Omega \varepsilon^{-d}\rho\left(\frac{|x-y|}{\varepsilon}\right)\frac{| \nabla  u_\varepsilon(x)- \nabla u_\varepsilon(y)|^2}{|x-y|^2 } \,\xd x\, \xd y,
\end{align}
where $C_p=C_p(d,\Omega)$ is a positive constant depending on the dimension $d\geq 1$, and on the domain $\Omega$, but independent of $\varepsilon$.\\

If $B_1<\frac{1}{2C_p}$, integrating \eqref{est1} in time yields for every $t\in [0,T]$
\begin{align*} 
\frac{1}{2}\|u_\varepsilon(t)\|_{L^2(\Omega)}^2&+\left(\frac{1}{2}-B_1C_p\right)\int_0^t \int_\Omega\int_\Omega \frac{\varepsilon^{-d}}{|x-y|^2}\rho\left(\frac{|x-y|}{\varepsilon}\right)| \nabla u_\varepsilon(x)-\nabla u_\varepsilon(y)|^2 \,\xd x\, \xd y \, \xd s \notag
\\
&\leq \frac{1}{2}\|u_0\|_{L^2(\Omega)}^2.
\end{align*}
Due to assumption \textbf{H4} on the initial data, recalling \eqref{ineq_ponce}, we get
\begin{align}
  \label{unif_1}
  \|u_\varepsilon\|_{L^\infty(0,T; L^2(\Omega))}
  +\|u_\varepsilon\|_{L^2(0,T; H^1(\Omega))}\leq C,\\
  \label{unif_3}
  \int_0^T\int_\Omega\int_\Omega \frac{\varepsilon^{-d}}{|x-y|^2}\rho\left(\frac{|x-y|}{\varepsilon}\right)|\nabla u_\varepsilon(x)-\nabla u_\varepsilon(y)|^2 \,\xd x\, \xd y \, \xd t\leq C
\end{align}
for some constant $C>0$ independent of $\varepsilon$.
Furthermore, from \eqref{unif_1} and \cite[Eq.~(5)]{ponce2004estimate} it also follows that
\begin{equation}
  \label{unif_2}
  \int_0^T\int_\Omega\int_\Omega \frac{\varepsilon^{-d}}{|x-y|^2}\rho\left(\frac{|x-y|}{\varepsilon}\right)|u_\varepsilon(x)- u_\varepsilon(y)|^2 \,\xd x\, \xd y \, \xd t\leq C\,.
\end{equation}
Now, setting
\[
B_\varepsilon(u)(x):=
(K_\varepsilon*1)u_\varepsilon-(K_\varepsilon*u_\varepsilon)=
\int_\Omega\varepsilon^{-d}\rho\left(\frac{|x-y|}{\varepsilon}\right)\frac{u_\varepsilon(x)-u_\varepsilon(y)}{|x-y|^2}\xd y,
\]
by the H\"older inequality 
and \cite[Eq.~(5)]{ponce2004estimate}
we get that for all $\psi\in H^1(\Omega)$,
\begin{align*}
&\langle B_\varepsilon(u_\varepsilon),\psi\rangle_{(H^{1}(\Omega))^*, H^1(\Omega)}=\frac{1}{2}\int_\Omega\int_\Omega K_\varepsilon(x,y)(u_\varepsilon(x)-u_\varepsilon(y))(\psi(x)-\psi(y))\xd y \xd x\\
&\leq\frac{1}{2}\Bigl(\int_\Omega\int_\Omega K_\varepsilon(x,y) |u_\varepsilon(x)-u_\varepsilon(y)|^2 \Bigr)^{1/2}\Bigl(\int_\Omega\int_\Omega K_\varepsilon(x,y) |\psi(x)-\psi(y)|^2 \Bigr)^{1/2}\\
&\leq C\Vert \nabla u_\varepsilon\Vert_{L^2(\Omega)}\Vert \nabla \psi\Vert_{L^2(\Omega)},
\end{align*}
where $C>0$ is some constant independent of $\varepsilon$ and $\psi$. Hence by \eqref{unif_1}
we also obtain  
\[
\Vert B_\varepsilon(u_\varepsilon)\Vert_{L^2(0,T;(H^{1}(\Omega))^*)}\leq C.
\]

Finally, testing the equation~\eqref{weak_NL}
by $(-\Delta)^{-1}\partial_t u_\varepsilon$
immediately yields, after integration in time,
\[
  \int_0^T\|\partial_t u_\varepsilon(t)\|_{(H^1(\Omega))^*}^2\,\xd t
  +E_\varepsilon(u_\varepsilon(T))
  =E_\varepsilon(u_{0,\varepsilon}),
\]
which implies thanks to \textbf{H4} that
\[
  \|\partial_t u_\varepsilon\|_{L^2(0,T;(H^1(\Omega))^*)}\leq C.
\]

%%%%%%%%%%%%%%%%%%%%%%%%%%%%%%%%%%%%%%%%%%%%%%%%%%%%%%%%%%%%%%%%%%%%%%%%%%%
\subsection{Convergence} \label{ssec:convergence}
%%%%%%%%%%%%%%%%%%%%%%%%%%%%%%%%%%%%%%%%%%%%%%%%%%%%%%%%%%%%%%%%%%%%%%%%%%%
Let $u_\varepsilon$ be the weak solution to the nonlocal Cahn-Hilliard equation~\eqref{eq:NLCH}.
Thanks to the uniform bounds derived above and the classical Aubin-Lions and Simon compactness results
(see \cite{simon}), we have the following convergences for the (not relabelled) subsequences:
\begin{align}
& u_\varepsilon \rightarrow u & \text{ strongly in } L^2(0,T;L^2(\Omega))\cap C^0([0,T]; (H^1(\Omega))^*) \label{strong_conv_u}, \\
& \partial_t u_\varepsilon \rightharpoonup \partial_t u & \text{ weakly}^*\text{ in } L^{2}(0,T;(H^1(\Omega))^\ast),\label{weak_conv_u_t}\\
& u_\varepsilon \rightharpoonup u & \text{ weakly* in } 
L^\infty(0,T; L^2(\Omega))\cap L^2(0,T; H^1(\Omega)),\\
& B_\varepsilon(u_\varepsilon)\rightharpoonup\xi & \text{ weakly* in } L^2(0,T;(H^1(\Omega))^*),
\label{weak_B}
\end{align}
for some $u \in L^2(0,T;H^1(\Omega)) \cap H^1(0,T;(H^1(\Omega))^\ast)$
and $\xi\in L^2(0,T; (H^1(\Omega))^*)$.

Let us show now that it also holds that 
\begin{equation}\label{strong_H1}
    u_\varepsilon \to u \quad\text{strongly in }
    L^2(0,T; H^1(\Omega)).
\end{equation}
In order to prove \eqref{strong_H1} we need the 
following lemma.

\begin{lemma}
For every $\delta>0$ there exist constants $C_{\delta}>0$ and $\varepsilon_\delta>0$ 
such that for every sequence $(f_\varepsilon)_\varepsilon\subset H^1(\Omega)$ there holds
\begin{align*}
   \label{eq:comp1}
    \|f_{\varepsilon_1}-f_{\varepsilon_2}\|^2_{H^1(\Omega)}&\leq \delta 
   \int_{\Omega}\int_{\Omega} \frac{\varepsilon_1^{-d}}{|x-y|^2}\rho\left(\frac{|x-y|}{\varepsilon_1}\right)
  |\nabla f_{\varepsilon_1}(x)-
  \nabla f_{\varepsilon_1}(y)|^2\xd x\xd y\\
  & +\delta 
   \int_{\Omega}\int_{\Omega} \frac{\varepsilon_2^{-d}}{|x-y|^2}\rho\left(\frac{|x-y|}{\varepsilon_2}\right)
  |\nabla f_{\varepsilon_2}(x)-
  \nabla f_{\varepsilon_2}(y)|^2\xd x\xd y\\
  &+C_{\delta}
  \|f_{\varepsilon_1}-f_{\varepsilon_2}\|^2_{L^2(\Omega)}
\end{align*}
for every pair $0<\varepsilon_1, \varepsilon_2 <\varepsilon_\delta$.
\end{lemma}
\begin{proof}
By contradiction suppose that there exists a $\bar{\delta}>0$ such that for every $n\in \mathbb{N}$ there is a sequence $(f_{\varepsilon}^n)_{\varepsilon}\subset H^1(\Omega)$ and parameters $\varepsilon_{1n},\varepsilon_{2n}<\frac1n$ such that
\begin{align*}
  \|f_{\varepsilon_{1n}}^n-f_{\varepsilon_{2n}}^n\|^2_{H^1(\Omega)} &> 
  \bar\delta
  \int_\Omega\int_\Omega \frac{\varepsilon_{1n}^{-d}}{|x-y|^2}\rho\left(\frac{|x-y|}{\varepsilon_{1n}}\right)|\nabla f_{\varepsilon_{1n}}^n(x)-\nabla f_{\varepsilon_{1n}}^n(y)|^2 \,\xd x\, \xd y\\
  &+\bar\delta
  \int_\Omega\int_\Omega \frac{\varepsilon_{2n}^{-d}}{|x-y|^2}\rho\left(\frac{|x-y|}{\varepsilon_{2n}}\right)|\nabla f_{\varepsilon_{2n}}^n(x)-\nabla f_{\varepsilon_{2n}}^n(y)|^2 \,\xd x\, \xd y\\
   &+n\|f_{\varepsilon_{1n}}^n-f_{\varepsilon_{2n}}^n\|^2_{L^2(\Omega)}\,.
\end{align*}
Noting that $\|f_{\varepsilon_{1n}}^n-f_{\varepsilon_{2n}}^n\|_{H^1(\Omega)}>0$
for every $n$ and setting
\[
g_{1n}:=\frac{f_{\varepsilon_{1n}}}{\|f_{\varepsilon_{1n}}-f_{\varepsilon_{2n}}\|_{H^1(\Omega)}}\,, \qquad g_{2n}:=\frac{f_{\varepsilon_{2n}}}{\|f_{\varepsilon_{1n}}-f_{\varepsilon_{2n}}\|_{H^1(\Omega)}}\,,
\]
we have
\begin{align*}
\bar\delta
  &\int_\Omega\int_\Omega \frac{\varepsilon_{1n}^{-d}}{|x-y|^2}
  \rho\left(\frac{|x-y|}{\varepsilon_{1n}}\right)
  |\nabla g_{1n}(x)-\nabla g_{1n}(y)|^2
  \,\xd x\, \xd y \\
  &+\bar\delta
  \int_\Omega\int_\Omega \frac{\varepsilon_{2n}^{-d}}{|x-y|^2}
  \rho\left(\frac{|x-y|}{\varepsilon_{2n}}\right)
  |\nabla g_{2n}(x)-\nabla g_{2n}(y)|^2
  \,\xd x\, \xd y
  + n \|g_{1n}-g_{2n}\|_{L^2(\Omega)}^2<1
  \quad\forall\,n\in\mathbb{N}.
\end{align*}
Such inequality immediately yields that $g_{1n}-g_{2n}\to0$ strongly in $L^2(\Omega)$
and the families $(\nabla g_{1n})_n$ and $(\nabla g_{2n})_n$
are relatively strongly compact in $L^2(\Omega;\mathbb{R}^d)$ by \cite[Thm.~1.2]{ponce2004estimate}.
We deduce that $g_{1n}-g_{2n}\to0$ strongly in
$H^1(\Omega)$, but this is a contradiction
since by definition we have $\|g_{1n}-g_{2n}\|_{H^1(\Omega)}=1$ for all $n$. 
\end{proof}

From the previous lemma it follows that for every 
$\delta>0$, there is $C_\delta >0$ and $\varepsilon_\delta>0$ 
such that for every pair 
$0<\varepsilon_1,\varepsilon_2 <\varepsilon_\delta$ such that 
\begin{align*}
  &
  \|u_{\varepsilon_1}-
  u_{\varepsilon_2}\|_{L^2(0,T;H^1(\Omega))}^2\\
  &\leq \delta\int_0^T
  \int_\Omega\int_\Omega \frac{\varepsilon_1^{-d}}{|x-y|^2}\rho\left(\frac{|x-y|}{\varepsilon_1}\right)|\nabla (u_{\varepsilon_1})(x)-\nabla (u_{\varepsilon_1})(y)|^2 \,\xd x\, \xd y\\
    &+ \delta\int_0^T
  \int_\Omega\int_\Omega \frac{\varepsilon_2^{-d}}{|x-y|^2}\rho\left(\frac{|x-y|}{\varepsilon_2}\right)|\nabla (u_{\varepsilon_2})(x)-\nabla (u_{\varepsilon_2})(y)|^2 \,\xd x\, \xd y\\
  &+C_\delta
  \|u_{\varepsilon_1}-
  u_{\varepsilon_2}\|_{L^2(0,T;L^2(\Omega))}^2.
\end{align*}
Thanks to the estimate \eqref{unif_3}, we infer that 
\[
  \|u_{\varepsilon_1}-u_{\varepsilon_2}\|_{L^2(0,T;H^1(\Omega))}^2
  \leq 2C\delta + C_\delta\|u_{\varepsilon_1}-u_{\varepsilon_2}\|_{L^2(0,T;L^2(\Omega))}^2
\]
for a certain constant $C>0$. Since $\delta$ is arbitrary 
and we already know that $u_\varepsilon\to u$ strongly in $L^2(0,T; L^2(\Omega))$, the strong convergence \eqref{strong_H1}
is proved.

We now prove the limit $u$ to be a weak solution of the local Cahn-Hilliard equation~\eqref{eq:CH}. We start from the Definition~\ref{def:wsnl} of weak solution for the nonlocal Cahn-Hilliard and we test it with a function $\varphi \in C^\infty([0,T]\times\Omega)$. By integrating by parts, we get 
\[
\int_0^T\int_\Omega (\partial_t u_\varepsilon) \varphi \, \xd x\, \xd t- \int_0^T \int_\Omega [(K_\varepsilon * 1) u_\varepsilon -K_\varepsilon * u_\varepsilon +F^{\prime}(u_\varepsilon)]\Delta \varphi\, \xd x\, \xd t =0.
\]

Now, let us show that $F'(u_\varepsilon)\to F'(u)$
in $L^1(0,T; L^1(\Omega))$. 
Indeed, by \textbf{H3}, the mean value theorem and the H\"older inequality we have that
\begin{align*}
&\int_0^T\int_\Omega|F'(u_\varepsilon)-F'(u)|
\leq
C\int_0^T\int_\Omega(1+|u|^2+|u_\varepsilon|^2)|u_\varepsilon-u|\\
&\leq
C\int_0^T\|u_\varepsilon-u\|_{L^6(\Omega)}
(1+\|u\|_{L^{12/5}(\Omega)}^2
+\|u_\varepsilon\|_{L^{12/5}(\Omega)}^2)\\
&\leq C\|u_\varepsilon-u\|_{L^2(0,T; H^1(\Omega))}\left(1+
\|u\|_{L^4(0,T; L^3(\Omega))}^2
+\|u_\varepsilon\|_{L^4(0,T; L^3(\Omega))}^2\right).
\end{align*}
The term in bracket on the right-hand side is bounded uniformly in $\varepsilon$
since by interpolation we have
\begin{align*}
  \|u_\varepsilon\|_{L^4(0,T; L^3(\Omega))}^2&\leq
  C\left(\int_0^T
  \|u_\varepsilon(t)\|^{2}_{L^6(\Omega)}
  \|u_\varepsilon\|^2_{L^2(\Omega)}\right)^{1/2}\\
  &\leq C\|u_\varepsilon\|_{L^\infty(0,T; L^2(\Omega))}\|u_\varepsilon\|_{L^2(0,T; H^1(\Omega))},
\end{align*}
and by
the estimate \eqref{unif_1}. Hence,
$F'(u_\varepsilon)\to F'(u)$
in $L^1(0,T; L^1(\Omega))$ by \eqref{strong_H1}.

Now,
by the convergence results above, we can pass to the limit in the variational formulation
and get
\[
\int_0^T\langle\partial_t u,\varphi\rangle_{(H^1(\Omega))^*,H^1(\Omega)} \, \xd t- 
\int_0^T 
\langle\xi,\Delta\varphi\rangle_{(H^1(\Omega))^*,H^1(\Omega)}\, \xd t
-\int_0^T\int_\Omega F^{\prime}(u)\Delta \varphi\, \xd x\, \xd t =0.
\]
It remains to identify the limit $\xi$ as $-\Delta u$.
To this end, we note that the variational derivative of the convex energy part
is given by
\[
\frac{\partial \tilde{E}_\varepsilon}{\partial u} = B_\varepsilon(u)
\]
and use the definition of subdifferential that reads
\begin{align}\label{eq:convex}
\tilde{E}_\varepsilon(z_1)+\langle B_\varepsilon(z_1),z_2-z_1\rangle_{(H^1(\Omega))^*,H^1(\Omega)} &\leq \tilde{E}_\varepsilon (z_2),
\end{align}
for all $z_1,z_2\in H^1(\Omega)$.
Hence, for all $z\in L^2(0,T; H^1(\Omega))$ we have that~\eqref{eq:convex} reads
\begin{equation}
    \int_0^T\tilde{E}_\varepsilon(u_\varepsilon)+\int_0^T\langle B_\varepsilon(u_\varepsilon),z-u_\varepsilon\rangle_{(H^{1}(\Omega))^*, H^1(\Omega)} \leq \int_0^T\tilde{E}_\varepsilon(z).
\end{equation}
The right hand side converges to $\int_0^T\tilde{E}(z)$ by~\cite[Eqs.~(3)\text{ and} (5)]{ponce2004estimate}
and the dominated convergence theorem, 
where 
\[
\tilde{E}(\cdot)=\frac12\int_\Omega|\nabla \cdot|^2\xd x.
\]
Moreover, thanks to the strong convergence~\eqref{strong_H1} and the 
weak convergence \eqref{weak_B}, we have that 
\[
\int_0^T\langle B_\varepsilon(u_\varepsilon),z-u_\varepsilon\rangle_{(H^{1}(\Omega))^*, H^1(\Omega)} \to 
\int_0^T\langle \xi,z-u\rangle_{(H^{1}(\Omega))^*, H^1(\Omega)}.
\] 
Next, we want to show that $\tilde{E}_\varepsilon(u_\varepsilon)\rightarrow \tilde{E}(u)$ in $L^1(0,T)$. 
In order to get this we first note that
\begin{align*}
&\tilde{E}_\varepsilon(u_\varepsilon)-\tilde{E}_\varepsilon(u)\\
&\leq \frac{1}{4}\int_\Omega\int_\Omega K_\varepsilon(x,y)\Bigl( (u_\varepsilon-u)(x)-(u_\varepsilon-u)(y)\Bigl)\Bigl( (u_\varepsilon+u)(x)-(u_\varepsilon+u)(y)\Bigl)\\
&\leq \frac{1}{4}\Bigl(\int_\Omega\int_\Omega K_\varepsilon(x,y) |(u_\varepsilon-u)(x)-(u_\varepsilon-u)(y)|^2 \Bigr)^{1/2}\\
&\quad\quad\Bigl(\int_\Omega\int_\Omega K_\varepsilon(x,y) |(u_\varepsilon+u)(x)-(u_\varepsilon+u)(y)|^2 \Bigr)^{1/2}\\
&\leq C\Vert \nabla (u_\varepsilon-u)\Vert_{L^2(\Omega)}\Vert \nabla (u_\varepsilon+u)\Vert_{L^2(\Omega)},
\end{align*}
where the first term converges to zero 
in $L^2(0,T)$ and the second one is bounded
in $L^2(0,T)$.
Hence $\tilde{E}_\varepsilon(u_\varepsilon)-\tilde{E}_\varepsilon(u)\to 0$ in $L^1(0,T)$,
and writing
\[
\tilde{E}_\varepsilon(u_\varepsilon)-\tilde{E}(u) =\Bigl( \tilde{E}_\varepsilon(u_\varepsilon)-\tilde{E}_\varepsilon(u)\Bigr) + \Bigl( \tilde{E}_\varepsilon(u)-\tilde{E}(u)\Bigr)\rightarrow 0
\]
we get the desired convergence as the second term on the right-hand side goes to zero in $L^1(0,T)$ thanks 
again to~\cite{ponce2004estimate}
and the dominated convergence theorem.

Hence, letting $\varepsilon\to0$
in the definition of the subdifferential of $\tilde{E}$, we deduce that 
\[
\int_0^T\tilde E(u)
+\int_0^T\langle \xi,z-u\rangle_{(H^{1}(\Omega))^*, H^1(\Omega)}
\leq \int_0^T\tilde E(z)
\]
for every $z\in L^2(0,T; H^1(\Omega))$, so that
$\xi=\partial \tilde{E}(u)=-\Delta u\in L^2(0,T;(H^1(\Omega))^*)$.

Thus, the limit $u$ satisfies an
equivalent integrated-in-time
variational formulation of the local
problem, that reads
\[
\int_0^T\int_\Omega (\partial_t u) \varphi \, \xd x\, \xd t- \int_0^T \int_\Omega \nabla u\cdot \nabla\Delta \varphi \, \xd x\, \xd t-\int_0^T \int_\Omega F^{\prime}(u)\Delta \varphi \, \xd x\, \xd t = 0
\]
for all $\varphi\in C^\infty([0,T]\times \Omega)$.
Note that the limit $u$ belongs to $L^2(0,T;H^2(\Omega))$ thanks to the result of Ponce~\cite{ponce2004estimate},
hence
we can integrate the second term by parts getting that $u$ satisfies equation~\eqref{eq def weak sol ch} in the sense of definition \ref{def:wsl}. 
This convergence result together with assumption \textbf{H4} and the density of $C^\infty(\Omega)$ in $H^2(\Omega)$ implies that $u$ is a weak solution to the Cahn-Hilliard equation~\eqref{eq:CH} and concludes the proof of Theorem~\ref{main thm}.

%%%%%%%%%%%%%%%%%%%%%%%%%%%%%%%%%%%%%%%%%%%%%%%%%
\section{Conclusions and Further Remarks}
\indent 
 %%%%%%%%%%%%%%%%%%%%%%%%%%%%%%%%%%%%%%%%%%%%%%%%%%%%%%%%%%%%%%%%%%%%%%%%%%%%%%%%%%%%%%%%%%%%%%%%%%%%%%%
In this paper we proved the convergence of weak solutions of the nonlocal Cahn-Hilliard equation~\eqref{eq:NLCH} to weak solutions of the local version~\eqref{eq:CH} as the convolution kernel $K$ approximates a Dirac delta in the case of periodic boundary conditions for dimension $d=3$. These conditions are physically relevant since the nonlocal energy functional has been derived starting from a lattice structure in the periodic setting, see~\cite{GL}. 
The proof uses the dynamic structure to obtain the appropriate estimates and regularity results. Moreover, an important key point in the proof is the application of an inequality in the spirit of Poincar\'e~\cite{ponce2004estimate} and the definition of subdifferential.\\

A natural question would be to investigate the case with other boundary conditions, as Dirichlet or Neumann conditions. Typically, boundary conditions (e.g. of Neumann type) for the local Cahn-Hilliard are imposed on the chemical potential $v$ as well as on $u$, whereas for the nonlocal Cahn-Hilliard are imposed only on $v$. 
Hence, it is not clear if, in the passage to the limit $\varepsilon \to 0$, (any type of) boundary conditions for $u$ has to be expected at all. Technically, this is related to difficulties in deriving uniform $H^1(\Omega)$ estimates and in proving $H^2$ regularity of the limit due to hard-to-handle boundary terms appearing when performing integration by parts.

%%%%%%%%%%%%%%%%%%%%%%%%%%%%%%%%%%%%%%%%%%%%%%%%%%%%%%%%%%%%%%%%%%%%%%%%%%%%%%%
\section*{Acknowledgements}
%%%%%%%%%%%%%%%%%%%%%%%%%%%%%%%%%%%%%%%%%%%%%%%%%%%%%%%%%%%%%%%%%%%%%%%%%%%%%%%

S.M.~was supported by the Austrian Science Fund (FWF) project P27052. H.R.~was funded by the Austrian Science Fund (FWF) project F 65. 
L.S.~was funded by Vienna Science and Technology Fund (WWTF) through Project MA14-009.
L.T.~acknowledges partial support from the Austrian Science Fund (FWF) projects F 65 and P27052.\\

\end{document}